\documentclass{amsart}
\usepackage{amssymb,amsmath,latexsym}
\usepackage{amsthm}
\usepackage{fontenc}
\usepackage{amssymb}

\numberwithin{equation}{section}

\newtheorem{theorem}{Theorem}[section]

\begin{document}
\author{Alexander E Patkowski}
\title{On Popov's formula involving the von Mangoldt function}

\maketitle
\begin{abstract}We offer a generalization of a formula of Popov involving the von Mangoldt function. Some commentary on its relation to other results in analytic number theory is mentioned as well as an analogue involving the M$\ddot{o}$bius function. \end{abstract}



\section{Introduction and Main result} In Titchmarsh's famous text on the Riemann zeta function [5], we find a reference to A.I. Popov's 1943 note [4]. Therein, we find the following curious formula, valid for $x>1:$
\begin{equation}\sum_{n>x}\frac{\Lambda(n)}{n^2}\left(\{\frac{n}{x}\}-\{\frac{n}{x}\}^2\right)=\frac{2-\log(2\pi)}{x}+\sum_{\rho}\frac{x^{\rho-2}}{\rho(\rho-1)}+\sum_{k\ge1}\frac{k+1-2k\zeta(2k+1)}{2k(k+1)(2k+1)}x^{-2k-2},\end{equation}
where $\Lambda(n)$ is the von Mangoldt function [3, p. 15, eq. (1.38)], $\zeta(s)$ is the Riemann zeta function [3, p. 12], and $\{x\}$ is the fractional part of $x.$ As usual, $\rho$ denotes the non-trivial zeros of the Riemann zeta function [2, p. 43]. Following this formula in [4] are some interesting corollaries, but ultimately no proof.\par

It is interesting to note that arithmetic series involving the fractional part function have also been studied by H. Davenport. We mention one of his identities [1, p. 8, eq. (C)]:
$$\sum_{n\ge1}\frac{\Lambda(n)}{n}\{nx\}=-\frac{1}{\pi}\sum_{n\ge1}\frac{\log(n)}{n}\sin(2\pi nx).$$

 Now it does not appear that a proof has been offered of (1.1) to the best of our knowledge. The purpose of this note is to offer a generalization of (1.1) and in doing so we offer what appears to be the first proof of Popov's result. The case $r=1$ gives (1.1) upon noting that $1-\rho$ is also a zero of $\zeta(s)$ [2, p. 18].

\begin{theorem} For real numbers $x>1$ and $r\ge1,$ we have

$$\frac{1}{2}\sum_{n>x}\frac{\Lambda(n)}{n^{r+1}}\left(\{\frac{n}{x}\}-\{\frac{n}{x}\}^2\right)=h_r(x)+\sum_{\rho}\left(\frac{r+1-\rho}{2(\rho+1-r)(\rho-r)}-\frac{\zeta(r-\rho)}{r-\rho} \right)\frac{x^{\rho-r-1}}{r+1-\rho}$$

\begin{equation}-\sum_{k\ge1}\left(\frac{2(k+1)+r-1}{2(1-2k-r)(2k+r)}+\frac{\zeta(2k+r)}{2k+r} \right)\frac{x^{-2k-r-1}}{2k+r+1},\end{equation}
where $$h_1(x)=\frac{2-\log(2\pi)}{2x},$$
and $r>1,$ $$h_r(x)=\left(\frac{r}{2(2-r)(1-r)}-\frac{\zeta(r-1)}{1-r}\right)\frac{x^{-r}}{r}.$$

\end{theorem}

\begin{proof} First, it is well-known that [5, p. 14, eq. (2.1.4)]
\begin{equation} \int_{1}^{\infty}t^{-s-1}\left(\{t\}-\frac{1}{2}\right)dt=\frac{s+1}{2s(s-1)}-\frac{\zeta(s)}{s},\end{equation}
for $\Re(s)>1.$ The integral converges because $\{t\}-\frac{1}{2}$ is bounded. Hence, using Mellin inversion and integrating over $[0, u],$ we have for $a>1,$
$$\frac{1}{2}\left(\{u\}^2-\{u\}\right)=\frac{1}{2\pi i}\int_{(a)}\left(\frac{s+1}{2s(s-1)}-\frac{\zeta(s)}{s}\right)\frac{u^{s+1}}{s+1}ds,$$
provided that $u>1.$ This may also be obtained by integration by parts and the absolute convergence of (1.3). Observe that the integrand does not have poles at $s=0$ or $s=1.$ Hence, for $-1<b<0,$
$$\frac{1}{2}\left(\{u\}^2-\{u\}\right)=\frac{1}{2\pi i}\int_{(b)}\left(\frac{s+1}{2s(s-1)}-\frac{\zeta(s)}{s}\right)\frac{u^{s+1}}{s+1}ds.$$ This integral may be seen to be absolutely convergent in this region through [5, p. 15, eq. (2.1.6)]. Now we may invert the desired series with $u=\frac{n}{x}>1,$ to obtain
$$\frac{1}{2}\sum_{n>x}\frac{\Lambda(n)}{n^{r+1}}\left(\{\frac{n}{x}\}-\{\frac{n}{x}\}^2\right)=\frac{1}{2\pi i}\int_{(b)}\left(\frac{s+1}{2s(s-1)}-\frac{\zeta(s)}{s}\right)\frac{x^{-s-1}\zeta'(r-s)}{\zeta(r-s)(s+1)}ds,$$
provided that $r\ge1,$ $r$ a real number. \par Replace $s$ with $1-s$ to get the integral for $1<c<2,$
\begin{equation}\frac{1}{2\pi i}\int_{(c)}\left(\frac{2-s}{2s(s-1)}-\frac{\zeta(1-s)}{1-s}\right)\frac{x^{s-2}\zeta'(s+r-1)}{\zeta(s+r-1)(2-s)}ds,\end{equation}
and we see there is a simple pole at $s=1$ when $r=1,$ a simple pole at $s=2-r$ when $r>1,$ simple pole at the non trivial zeros $s=\rho+1-r,$ and simple pole at the trivial zeros $s=-2k-r+1.$ The residue at the simple pole $s=1,$ $r=1$ is the $h_1(x),$ and the residue at the simple pole $s=2-r$ when $r>1,$ is the $h_r(x),$ $r>1,$ part of the theorem. The residue at the non-trivial zeros $s=\rho+1-r$ gives the sum over $\rho$ in the theorem. Lastly, the residue at the trivial zeros $s=-2k-r+1$ gives the sum over $k$ on the far right side of the theorem. After these residues have been computed the remaining integral is $0$ and the theorem follows. \end{proof}

\section{Further Comments}
We do not know specifically where Popov encountered (1.1), but we believe it is interesting to note that the sum involving the roots $\rho$ appears in [2, p. 69], which shows a connection with $\psi(x)=\sum_{n\le x}\Lambda(n),$ through the integral $\int_{0}^{x}\psi(t)dt/t^2.$ It is not difficult to make small adjustments to our proof to include other arithmetic functions as well. In closing we offer the analogue for Merten's function $M(x)=\sum_{n\le x}\mu(n)$ [5, p. 370], involving $\int_{0}^{x}M(t)dt/t^2.$
\begin{theorem} Assume all the non-trivial zeros of the Riemann zeta function are simple. We have for $x>1,$
\begin{equation}-\sum_{n>x}\frac{\mu(n)}{n^2}\left(\{\frac{n}{x}\}-\{\frac{n}{x}\}^2\right)=\sum_{\rho}\frac{x^{\rho-2}}{\zeta'(\rho)\rho(\rho-1)}+\sum_{k\ge1}\frac{k+1-2k\zeta(2k+1)}{2k(k+1)(2k+1)}\frac{\pi^{2k}2^{2k+1}(-1)^k}{(2k)!\zeta(2k+1)}x^{-2k-2}.\end{equation}

\end{theorem}

\begin{proof} The proof is identical to the $r=1$ case of Theorem 1.1 with the notable difference that the residue at the pole $s=1$ is 0, and at the trivial zeros we apply the formula
$$\frac{1}{\zeta'(-2n)}=\frac{\pi^{2n}2^{2n+1}}{(-1)^n\zeta(2n+1)(2n)!},$$ for natural numbers $n>0.$ This is easily found from differentiating the functional equation [5, p. 16, eq. (2.1.8)] 
$$\zeta(1-s)=2^{1-s}\pi^{-s}\cos(\frac{\pi}{2}s)\Gamma(s)\zeta(s),$$ with respect to $s$ and setting $s=2n+1.$ The remaining details are left for the reader. \end{proof}

1390 Bumps River Rd. \\*
Centerville, MA
02632 \\*
USA \\*
E-mail: alexpatk@hotmail.com
\end{document}